\makeatletter
\newcommand{\input@path}{%
  {../data/}%
  {../associative_theories/}%
  {../associative_theories/Countermodels/}%
  {../}}\makeatother

\documentclass[12pt]{amsart}
\usepackage[utf8]{inputenc}
\usepackage[T1]{fontenc}
\usepackage{lmodern}
\usepackage{amsmath, amssymb, amsthm, mathtools}
\usepackage{geometry}
\usepackage{xcolor, graphicx}
\usepackage{array, booktabs}
\usepackage{needspace}

\DeclareUnicodeCharacter{221E}{\ensuremath{\infty}}

\usepackage{listings}
\lstnewenvironment{lean}{\lstset{language=lean}}{}
\definecolor{kwcolor}{rgb}{0.8, 0.1, 0.1}    
\definecolor{taccolor}{rgb}{0.1, 0.1, 0.8}    
\definecolor{opcolor}{rgb}{0.5, 0.1, 0.5}    
\newcommand{\leancomment}{\color{gray}}
\lstdefinelanguage{lean} {
  basicstyle = {\scriptsize \ttfamily},
  columns = fullflexible,
  keepspaces = true,
  morecomment = [l][\leancomment]{--},
  morecomment = [s][\color{green!50!black}]{@[}{]},
  morekeywords = [1]{theorem, by, fun, let, have, obtain},
  keywordstyle = [1]{\ttfamily\color{kwcolor}},
  morekeywords = [2]{use, constructor, intro, only, simp, decide, decideFin, by_contra, subsumption},
  keywordstyle = [2]{\ttfamily\color{taccolor}},
  literate =
  {:=}{{{\color{kwcolor}\raise1pt\hbox{:}\hspace{-1pt}=}}}{2}
  {=>}{{{\color{kwcolor}=>}}}{2}
  {∀}{{{\color{opcolor}\ensuremath{\forall}}}}{1}
  {∃}{{{\color{opcolor}\ensuremath{\exists}}}}{1}
  {∧}{{{\color{opcolor}\ensuremath{\wedge}}}}{1}
  {¬}{{{\color{opcolor}\ensuremath{\lnot}}}}{1}
  {×}{{{\color{opcolor}\ensuremath{\times}}}}{1}
  {◇}{{{\color{opcolor}\ensuremath{\diamond}}}}{1}
  {·}{{{\color{taccolor}\ensuremath{\bullet}}}}{1}
  {≠}{{{\color{opcolor}\ensuremath{\neq}}}}{1}
  {ℕ}{{\ensuremath{\mathbb{N}}}}{1}
  {⟨}{{\ensuremath{\langle}}}{1}
  {⟩}{{\ensuremath{\rangle}}}{1}
  {«}{\color{gray}{\guillemetleft}}{1}
  {»}{{\color{gray}\guillemetright}\color{black}}{1}
  ,
}
\usepackage{iftex}
\ifpdftex\else
\usepackage{newunicodechar}
\newunicodechar{∃}{\ensuremath{\exists}}
\newunicodechar{∧}{\ensuremath{\wedge}}
\newunicodechar{¬}{\ensuremath{\lnot}}
\newunicodechar{×}{\ensuremath{\times}}
\newunicodechar{◇}{\ensuremath{\diamond}}
\newunicodechar{·}{\ensuremath{\bullet}}
\newunicodechar{≠}{\ensuremath{\neq}}
\newunicodechar{ℕ}{\ensuremath{\mathbb{N}}}
\newunicodechar{⟨}{\ensuremath{\langle}}
\newunicodechar{⟩}{\ensuremath{\rangle}}
\newunicodechar{«}{\guillemetleft}
\newunicodechar{»}{\guillemetright}
\newunicodechar{–}{\textendash}
\newunicodechar{’}{\textquoteright}
\newunicodechar{“}{\textquotedblleft}
\newunicodechar{”}{\textquotedblright}
\newunicodechar{š}{\v s}
\fi

\usepackage[hidelinks]{hyperref}
\usepackage{cleveref}
\usepackage{aliascnt}

\theoremstyle{plain}
\newtheorem{theorem}{Theorem}[section]
\renewcommand{\theHtheorem}{\theHsection.\the\value{theorem}} 
\newcommand{\mynewtheorem}[2]{
  \newaliascnt{#1}{theorem}
  \newtheorem{#1}[#1]{#2}
  \aliascntresetthe{#1}
  \expandafter\providecommand\csname #1autorefname\endcsname{#2}
}
\mynewtheorem{proposition}{Proposition}
\mynewtheorem{lemma}{Lemma}
\mynewtheorem{corollary}{Corollary}
\mynewtheorem{definition}{Definition}
\theoremstyle{definition}
\mynewtheorem{example}{Example}
\mynewtheorem{remark}{Remark}

\newcommand{\op}{\diamond}
\newcommand{\implicitop}{}

\newcommand{\Equaref}[1]{\href{https://teorth.github.io/equational_theories/implications/?#1}{\textrm{\textup{E}}#1}}
\newcommand{\Equadef}[2]{\expandafter\gdef\csname Equa#1\endcsname{\ensuremath{#2}}}
\newcommand{\Equareffull}[1]{\Equaref{#1}\ifmmode\,\else\ \fi\ensuremath{(\csname Equa#1\endcsname)}}
\DeclareMathOperator{\hen}{hen}

\begingroup
\ExplSyntaxOn
\exp_args_generate:n { xV }
\ior_new:N \g_my_ior
\ior_open:Nn \g_my_ior { associative_equations.txt }
\ior_map_inline:Nn \g_my_ior
  {
    \seq_set_split:Nnn \l_tmpa_seq { : } {#1}
    \tl_set:Nx \l_tmpb_tl { \seq_item:Nn \l_tmpa_seq { 2 } }
    \tl_replace_all:Nnn \l_tmpb_tl { ◇ } { \implicitop }
    \exp_args:NxV \Equadef { \seq_item:Nn \l_tmpa_seq { 1 } } \l_tmpb_tl
  }
\endgroup

\title[Lattice of 456 semigroup varieties]{Lattice of 456 semigroup varieties\\from equations of order up to 4}
\author{Bruno Le Floch}
\address{CNRS and Laboratoire de Physique Th\'eorique et Hautes \'Energies, Sorbonne Universit\'e, blefloch@lpthe.jussieu.fr, ORCID 0000-0002-3965-9705}
\date{September 2026}

\begin{document}

\begin{abstract}
  We consider the 653 equational laws of order up to 4 for an associative binary operation, and all of their conjunctions.
  We determine that there are only 456 equivalence classes of such conjunctions (associative equational theories), and find all implications between them.
  The conjunction operation makes this set of theories into a semi-lattice.
  These results are formalized in Lean.
  This is a semigroup analogue of the Equational Theories Project, but extended to conjunctions of equations.
\end{abstract}

\begingroup
\def\uppercasenonmath#1{} 
\let\MakeUppercase\relax 
\maketitle
\endgroup

\setcounter{tocdepth}{2}
\tableofcontents

\section{Introduction and conclusions}

\subsection{Semigroup equational laws}

Many algebraic structures can be defined as a set equipped with unary, binary or other operations on that set, subject to certain equations such as associativity or commutativity.  Here we concentrate on \emph{semigroups}, which are sets~$S$ equipped with a binary operation $\op:S\times S\to S$ that is associative, namely
\begin{equation}\label{eq4512-intro}
  \forall x,y,z\in S , \ x \op (y \op z) = (x \op y) \op z ,
\end{equation}
and we ask what structures can be described as semigroups subject to further equations, with every variable being subject to a universal quantifier, as in~\eqref{eq4512-intro}.

Associativity allows us to omit parentheses, and we also omit the~$\op$ symbol for brevity.
Throughout this work we will \emph{leave universal quantifiers implicit}.
As a prototypical example, a \emph{band} is a semigroup where the associative operation~$\op$ obeys the idempotence law (implicitly for all $x\in S$)
\begin{equation}\label{eq3-intro}
  x = x \implicitop x .
\end{equation}
We determine all implications between identities of order (number of operations) up to~$4$ and their arbitrary conjunctions.
Specifically, the outcome of this work is the following theorem.

\needspace{6\baselineskip}
\begin{theorem}
  Any associative equational law of order at most~$4$ is equivalent to one of $122$ laws, listed in \autoref{app:allequations}.
  Any conjunction of such laws sits in one of $456$ equivalence classes, listed in \autoref{app:alltheories} as minimal and maximal subsets of the~$122$ laws.
  In terms of these maximal subsets, the order and the meet operation on semigroup varieties is given by the subset relation and the union.
  These results are unchanged when restricting to finite semigroups.
\end{theorem}

We explain in \autoref{sec:informal} how the lattice was determined using the Prover9/Mace4 ATP\@, and describe how the \emph{minimal} representative of each equivalence class was selected.
To avoid introducing an ad-hoc numbering system, we adopt throughout the numbering conventions of the Equational Theories Project (ETP) introduced for magma equations, which results in skipping numerous equation numbers, as explained in \autoref{subsec:numbering}.
The explicit list of equivalence classes ($122$~laws and $456$~conjunctions) are given in \autoref{app:supplementary}, and as machine-readable data files available in the repository \url{https://github.com/blefloch/associative_theories/} described in \autoref{subsec:repository}.

All implications and conjunctions stated in the theorem are formalized in Lean, but we have not yet written an end-to-end theorem, as explained in \autoref{sec:lean}.
As in the ETP, the choice of Lean is rather arbitrary and stems solely from our limited knowledge of other formalization tools such as Rocq.
Most of the Lean code is generated by Python scripts rather than through clever Lean metaprogramming.  This leads to a rather bloated 45\,MB proof, which is locally quite readable but of course not fully human readable due to its length.  There were difficulties whereby even the \texttt{tauto} tactic was too slow to manipulate large conjunctions.

We end in \autoref{sec:future} by describing some links to other ongoing work and future work.

For now, after briefly reviewing in \autoref{subsec:varieties} the notion of lattice of semigroup variety and some earlier literature, we comment in \autoref{subsec:auto-vs-form} on the crucial distinction between automation and formalization.

\subsection{Varieties of semigroups}
\label{subsec:varieties}

In the language of universal algebra, a \emph{variety of semigroups} consists of a class of semigroups that can be characterized by a (possibly infinite) conjunction of such equations.  Thus, any intersection of varieties is a variety, characterized by the conjunction of all defining equations of the varieties of interest.
The set of all varieties of semigroups forms a lattice:
\begin{itemize}
\item the (partial) order is given by inclusion of varieties;
\item the meet operation is the intersection of varieties, which is defined for any collection of varieties (finite or infinite);
\item it admits the variety of all semigroups as its top element;
\item thus the join of two varieties can be defined as the meet of all their upper bounds, namely the smallest variety containing both varieties.
\end{itemize}

The lattice of all semigroup varieties is highly complicated; for instance it contains the partition lattice of a countably-infinite set~\cite{burris1971embedding,jevzek1976intervals}.  It has seen very substantive literature which we can only mention very superficially.
Reviews from 1971~\cite{evans1971lattice} and 2009~\cite{shevrin2009lattices} are available.
An important question is to find elements in this lattice with particular properties: atoms, modular elements, etc.~\cite{vernikov2007modular,vernikov2015special,skokov2025modular}.

Some sublattices admit explicit descriptions.
The lattice of \emph{band varieties} (idempotent semigroups) was fully described in 1970~\cite{biryukov1970varieties,fennemore1970all,GERHARD1970195}, with some later developments~\cite{wismath1986lattices,https://doi.org/10.1112/plms/s3-58.2.323,ghosh2005varieties,sapir2005variety,pastijn2005varieties}.
\emph{Completely regular semigroup} varieties were studied throughout the decades since at least the 80's~\cite{petrich1982varieties,jones1983lattice,reilly1985varieties,pastijn1985lattices,polak1985varieties,reilly1987lattice,polak1987varieties,polak1988varieties,trotter1989subdirect,pastijn1990lattice,reilly1990completely,pastijn1991pseudovarieties,zhang1996infinite,petrich2007canonical,Petrich03042014,petrich2015varieties,petrich2017another,petrich2018certain,petrich2024completely}.
The lattice of varieties of \emph{commutative semigroups} was only fully understood in the 90's~\cite{kisielewicz1994varieties}.
Many other sublattices were considered, cf.~\cite{petrich1974all,dolinka2001lattice,gusev2020lattice}.

Instead of treating equational laws as the basic building blocks characterizing varieties, a different starting point is to pick a semigroup~$S$, consider all equational laws it obeys, and the variety of semigroups satisfying all of these same laws.  See \cite{luo2011variety} for some results for $2,3$ elements.
This was considered for semigroups with up to~$5$ elements in~\cite{ARAUJO2023698}, whose authors provide a \emph{very extensive} set of calculational tools for semigroup varieties, as well as automated ways to explore the literature.
Since some non-implications in our work require counterexamples of order up to~$32$, the varieties we find are likely not all covered by that work.  It would nevertheless be desirable to match some of the $456$ varieties considered in the present work to those of~\cite{luo2011variety,ARAUJO2023698}.

\subsection{Automation vs formalization}
\label{subsec:auto-vs-form}

We find it relevant to distinguish between two axes: manual--automated and informal--formalized.  All four combinations were relevant in this work.
\begin{itemize}
\item \textit{Manual and informal:} deciding the numbering system (\autoref{subsec:numbering}) and how to represent equivalence classes of conjunctions (\autoref{subsec:representative-choice}); organizing how equivalence classes should be sought (\autoref{subsec:main-search}); writing the paper.

\item \textit{Manual and formalized:} setting up the Lean environment (\autoref{subsec:repository}, \autoref{fig:lakefile}), writing individual Lean theorems such as the main theorem \texttt{AT\_conj} (\autoref{sec:lean}).

\item \textit{Automated and informal:} calling the Mace4/Prover9 ATPs through Python scripts to initially get the implication graph (\autoref{subsec:main-search}).

\item \textit{Automated and formal:} Vampire--Lean integration for the equational reasoning; logical implication steps done by writing Lean code via Python scripts.
\end{itemize}

\subsection*{Acknowledgements}

The author thanks Equational Theories Project collaborators for a motivating and fun research setting, and especially Jose Brox for collaboration at early stages of this project.
The author is partially supported by the grants Projet-ANR-23-CE40-0010 and HORIZON-MSCA-2022-SE/101131233.

\subsection*{AI use statement}

ChatGPT (circa GPT-5) was used to suggest Lean syntax instead of consulting the Lean documentation, and to clarify a confusion about lattices.
Google Scholar and analogous bibliography search tools likely rely on some underlying machine learning.
It is possible that the search strategy used by the automated theorem provers Vampire and Prover9/Mace4 involves some learning.
No other AI was used in this project.

\section{Determining the implication graph informally}
\label{sec:informal}

\subsection{Numbering}
\label{subsec:numbering}

Since we manipulate somewhat large collections of laws, it is useful to number them.  Rather than introducing an ad-hoc numbering and having to provide a dictionary with other numbering systems, we adopt equation numbers introduced in the Equational Theories Project (ETP\@).  There, the equations of interest are of the form $w_1=w_2$ with $w_1,w_2$ constructed from some (universally quantified) variables $x_0,x_1,\dots$, a single binary operation~$\op$, and parentheses.  Equations are ordered by considering first their total order (number of appearances of~$\op$), then their shape (ordered recursively lexicographically), then lexicographically in the variables.
Tautological equations, namely equations whose left-hand side and right-hand side are exactly identical (with the same variables) are omitted, except for the equation $x_0=x_0$.
Equations that only differ by the exchange of the left-hand side and right-hand side, or by a relabelling of variables, are identified, and only the first of these (in the given order) is retained.
The resulting list of equations is then numbered consecutively
{\renewcommand{\implicitop}{\op}%
\Equareffull{1}, \Equareffull{2}, \Equareffull{3}, \Equareffull{4}, \Equareffull{5}, \Equareffull{6}, \dots}

For semigroup equations the situation is simpler as there is no shape to consider, beyond the number of~$\op$ on each side of the equation.  Thus, an equation $w_1=w_2$ is ordered before $w_1'=w_2'$ if
\begin{itemize}
\item its order is strictly less, or
\item its order is equal and the order of $w_1$ is strictly less than that of $w_1'$, or
\item the orders of $w_1$ and $w_1'$ are equal, the orders of $w_2$ and $w_2'$ are equal, and one has $w_1w_2=x_{i_1}\dots x_{i_k}$ and $w_1'w_2'=x_{j_1}\dots x_{j_k}$ with $(i_1,\dots,i_k)<(j_1,\dots,j_k)$ lexicographically.
\end{itemize}
As in the ETP, tautological equations other than $x_0=x_0$ are omitted, and only the minimal equation under exchanging the sides and relabelling variables is retained.
To avoid clutter, the variables are relabeled from $x_0,x_1,\dots$ to $x,y,z,\dots$ as needed.

The list of $31$ semigroup equations of order up to~$2$, ordered in this way, is thus
\begin{equation}
  \begin{aligned}
    \Equareffull{1},
    \Equareffull{2},
    \Equareffull{3},
    \Equareffull{4},
    \Equareffull{5},
    \Equareffull{6},
    \\
    \Equareffull{7},
    \Equareffull{8},
    \Equareffull{9},
    \Equareffull{10},
    \Equareffull{11},
    \\
    \Equareffull{12},
    \Equareffull{13},
    \Equareffull{14},
    \Equareffull{15},
    \Equareffull{16},
    \\
    \Equareffull{17},
    \Equareffull{18},
    \Equareffull{19},
    \Equareffull{20},
    \Equareffull{21},
    \\
    \Equareffull{22},
    \Equareffull{38},
    \Equareffull{39},
    \Equareffull{40},
    \Equareffull{41},
    \\
    \Equareffull{42},
    \Equareffull{43},
    \Equareffull{44},
    \Equareffull{45},
    \Equareffull{46}.
  \end{aligned}
\end{equation}
Observe the absence of the equation numbers E23 to E37: without associativity, these are equations of the form $x = (\_ \op \_) \op \_$ with $\_$ standing for various variables depending on the equation.  Associativity makes them immediately equivalent to an equation of the form $x = \_ \op (\_ \op \_)$.  More generally, the only magma equations that remain in the semigroup context are those whose left-hand side and right-hand side are right-associated.
The full list of semigroup equations of order up to~$4$ is given in \autoref{app:allequations}, organized according to their equivalence classes.
The number of associative equations of different orders, and the number of equivalence classes whose shortest representative has different orders, is
\begin{center}
  \begin{tabular}{lrrrrrr}
    \toprule
    order & 0 & 1 & 2 & 3 & 4 & total \\
    \midrule
    equations & 2 & 5 & 24 & 104 & 518 & 653 \\
    classes & 2 & 3 & 10 & 23 & 84 & 122 \\
    \bottomrule
  \end{tabular}
\end{center}

It is worth mentioning a potential pitfall when converting from magma equations to semigroup equations, in the case where the two sides have the same order. Consider as an example the magma equation \Equaref{4489} ($x \op (y \op y) = (z \op x) \op x$).  Under associativity, it reduces to the equation $x \op (y \op y) = z \op (x \op x)$, but as such, this equation is \emph{not} in the list of equations considered in the ETP: one must swap the two sides and relabel the variables to get \Equaref{4346} ($x \op (y \op y) = y \op (z \op z)$).
This led to a subtle bug in Python scripts at early stages of the project, where the inexistent equation $x \op (y \op y) = z \op (x \op x)$ was erroneously constructed as a Python object, and its ETP equation number was computed using code that assumed valid equations only.  A similar difficulty arises when dualizing equations, where simply reversing the left-hand side and the right-hand side does not ensure that they are in the correct order.

\subsection{Choice of representative}
\label{subsec:representative-choice}

The $653$ semigroup equations of order up to~$4$ reduce to~$122$ equivalence classes.  Each of the $2^{122}$ subsets of these equations could in principle define a different semigroup variety, but there are a huge number of equivalences, such as
\begin{equation}
  \begin{aligned}
    \Equaref{2}\ \wedge\ \text{anything} & \iff \Equareffull{2}, \\
    \Equareffull{38}\ \wedge\ \Equareffull{39} & \iff \Equareffull{41} .
  \end{aligned}
\end{equation}
Such equivalences drastically reduce the number of equivalence classes of conjunctions, down to only~$456$.
Almost all of these equivalence classes can thus be represented in numerous manners as conjunctions of subsets of the $122$~equations.  There are several methods to single out one choice.  We prefer the ``early representative'', as it is insensitive to cutting off the maximum equation order to~$4$, as proven in \autoref{prop:early-cutoff}.  This gives us a universally-defined way to number associative theories, as done in \autoref{app:alltheories}.

\begin{definition}\label{def:representatives}
  The \textbf{long representative (truncated to order~$k$)} of a conjunction of equations is the set of all equations of order up to~$k$ that it implies.

  The \textbf{minimal-equations representative} is obtained by selecting in the long representative the smallest number of equations whose conjunction is in the desired equivalence class; if multiple choices are available, choose the earliest in lexicographic order.

  The \textbf{early representative} is obtained by repeatedly removing from the long representative the highest-numbered equation that can be removed without changing the equivalence class.
\end{definition}

\begin{remark}
In \autoref{app:alltheories}, we list the $456$ classes by giving the early and the long representatives for each one.  We have not computed minimal-equations representatives in general.  In \autoref{tab:example-classes}, we give a few examples of equivalence classes.  While the early representatitive is canonical once the ordering of equations has been chosen, it often obscures features of the equivalence classes: most manifestly, the two projection laws \Equareffull{38}${}\wedge{}$\Equareffull{39} are simply equivalent to the constant law \Equareffull{41}.
Sometimes it is nevertheless clearer than the minimal-equations representative, which packs several pieces of information together: for instance, \Equareffull{43}${}\wedge{}$\Equareffull{56} gives commutativity and unit cubes, which are not manifest in the single equation \Equareffull{66} ---however both descriptions obscure the fact that cubes are all equal.
\end{remark}

\begin{table}
\caption{\label{tab:example-classes}Examples of equivalence classes of conjunctions and possible choices of representatives.  We recall \Equareffull{40} and \Equareffull{43}, which we do not explicit below for brevity.}
\centering\scriptsize
\begin{tabular}{lll}
  \toprule
  Early representative & Minimal-equations & Long representative \\
  \Equareffull{38}${}\wedge{}$\Equareffull{39} & \Equareffull{41} & $\{1, 38, 39, 40, 41\dots\}$ \\
  \Equaref{43}${}\wedge{}$\Equareffull{56} & \Equareffull{66} & $\{1, 43, 47, 56, 66\dots\}$ \\
  \Equaref{40}${}\wedge{}$\Equareffull{3306} & \Equareffull{3350} & $\{1, 40, 43, 3253\dots\}$ \\
  \Equaref{40}${}\wedge{}$\Equareffull{307} & \Equareffull{316} & $\{1, 40, 307, 308\dots\}$ \\
  \Equareffull{323}${}\wedge{}$\Equareffull{326} & \Equareffull{3309} & $\{1, 307, 323, 326\dots\}$ \\
  \Equaref{40}${}\wedge{}$\Equaref{43}${}\wedge{}$\Equareffull{307} & \Equaref{43}${}\wedge{}$\Equareffull{3255} & $\{1, 40, 43, 307, 308\dots\}$ \\
  \bottomrule
\end{tabular}
\end{table}

\begin{proposition}\label{prop:early-cutoff}
  Define a total order on finite sets~$\Sigma$ of semigroup equations by declaring that $\Sigma\leq\Sigma'$ if either of the following equivalent characterizations hold:
  \begin{itemize}
  \item for any equation in $\Sigma\setminus\Sigma'$ there exists a higher-numbered equation in $\Sigma'\setminus\Sigma$;
  \item either $\Sigma\subseteq\Sigma'$, or neither set is contained in the other and the highest-numbered equation of $\Sigma\setminus\Sigma'$ is less (in the ETP order) than the highest-numbered equation of $\Sigma'\setminus\Sigma$.
  \end{itemize}
  This is a lexicographic order in decreasing equation numbers.
  The early representative of an equivalence class of conjunctions is then the minimum subset~$\Sigma$ (for this order) whose conjunction is in the equivalence class.
  In particular, the early representative is unaffected by the cutoff on the equation order.
\end{proposition}

\begin{proof}
  {\it Step 1. Total order.}
  For brevity, we denote by $\hen(\Sigma)$ the highest-equation number in the set~$\Sigma$.
  Let us first check that the two definitions of~$\leq$ are equivalent.  If $\Sigma\setminus\Sigma'=\emptyset$ then the first definition is vacuuous while the second holds thanks to $\Sigma\subseteq\Sigma'$.  Otherwise the condition that every equation in $\Sigma\setminus\Sigma'$ is less than an equation in $\Sigma'\setminus\Sigma$ implies that $\hen(\Sigma\setminus\Sigma')$ is less than some equation number in $\Sigma'\setminus\Sigma$, hence than $\hen(\Sigma'\setminus\Sigma)$, while conversely any equation in $\Sigma\setminus\Sigma'$ has a number ${}\leq\hen(\Sigma\setminus\Sigma')<\hen(\Sigma'\setminus\Sigma)$.

  The relation~$\leq$ is antisymmetric and total since for distinct sets $\Sigma,\Sigma'$, either one is contained in the other, or neither is, in which case $\Sigma\setminus\Sigma'$ and $\Sigma'\setminus\Sigma$ are non-empty disjoint subsets, whose highest-number equations are thus distinct and comparable by the total order on equations defined by the Equational Theories Project.
  Transitivity in $\Sigma\leq\Sigma'\leq\Sigma''$ is shown by distinguishing four cases.  First, if $\Sigma\subseteq\Sigma'\subseteq\Sigma''$ then $\Sigma\subseteq\Sigma''$.  Second, if $\Sigma\subseteq\Sigma'$ and $\hen(\Sigma'\setminus\Sigma'')<\hen(\Sigma''\setminus\Sigma')$ then by virtue of set inclusions $\Sigma\setminus\Sigma''\subset\Sigma'\setminus\Sigma''$ and $\Sigma''\setminus\Sigma'\subset\Sigma''\setminus\Sigma$ one gets
\begin{equation}
\hen(\Sigma\setminus\Sigma'')\leq\hen(\Sigma'\setminus\Sigma'')<\hen(\Sigma''\setminus\Sigma')\leq\hen(\Sigma''\setminus\Sigma) ;
\end{equation}
Third, if $\hen(\Sigma\setminus\Sigma')<\hen(\Sigma'\setminus\Sigma)$ and $\Sigma'\subseteq\Sigma''$ then again by set inclusions, $\hen(\Sigma\setminus\Sigma'')<\hen(\Sigma''\setminus\Sigma)$.  Fourth, if finally
\begin{equation}\label{hen-hen-assume}
  \hen(\Sigma\setminus\Sigma')<\hen(\Sigma'\setminus\Sigma) \text{ and }
  \hen(\Sigma'\setminus\Sigma'')<\hen(\Sigma''\setminus\Sigma')
\end{equation}
we consider three disjoint sets~$\Psi_i$ and their union,
\begin{equation}
  \begin{aligned}
    \Psi_1 & \coloneqq \Sigma\setminus\Sigma' , \quad \Psi_2 \coloneqq \Sigma'\setminus\Sigma'' , \quad \Psi_3 \coloneqq \Sigma''\setminus\Sigma ,
    \\
    \Psi & \coloneqq \Psi_1 \sqcup \Psi_2 \sqcup \Psi_3 = (\Sigma\cup\Sigma'\cup\Sigma'')\setminus(\Sigma\cap\Sigma'\cap\Sigma'') ,
  \end{aligned}
\end{equation}
and start by locating the highest-numbered equation $E$ in $\Psi\coloneqq$.
The assumptions in~\eqref{hen-hen-assume} forbid~$E$ from being in the first two of these subsets, hence $E$ is in $\Psi_3=\Sigma''\setminus\Sigma$.  As a result, $\hen(\Sigma\setminus\Sigma'')\leq\hen(\Psi_1\sqcup\Psi_2)<\hen(\Psi_3)$.
This concludes the proof that $\leq$ is a total order.

\medskip

{\it Step 2. Equivalence with \autoref{def:representatives}.}
Consider an equivalence class of (finite) conjunctions of equations, and fix an order cutoff~$k$ which is large enough that the equivalence class admits at least one representative consisting only of equations of order at most~$k$.  Denote by $\Lambda_k$ the long representative truncated to order~$k$ and by~$\Xi_k$ the early representative defined in \autoref{def:representatives}, obtained by greedily removing equations from~$\Lambda_k$ starting from the largest equation numbers (while keeping the conjunction's equivalence class fixed).  Our aim is to show that for any representative~$\Sigma$ one has $\Xi_k\leq\Sigma$ in the order analyzed in step~1.  Assume by contradiction that there exists a representative with $\Sigma<\Xi_k$.

One cannot have $\Sigma\subsetneq\Xi_k$ as the process defining the early representative would have continued removing equations from~$\Xi_k$.  We deduce that $\hen(\Sigma\setminus\Xi_k) < \hen(\Xi_k\setminus\Sigma)$.  Since all equations in~$\Xi_k$ have order at most~$k$, this inequality implies that all equations in $\Sigma\setminus\Xi_k$ (hence in all of~$\Sigma$) have order at most~$k$, and as a result $\Sigma\subset\Lambda_k$.

Denote now by~$E$ the highest-numbered equation of $\Xi_k\setminus\Sigma$, and split $\Lambda_k=\Lambda_k^-\sqcup\{E\}\sqcup\Lambda_k^+$ into equations whose ETP number is less than, equal to, or larger than that of~$E$, respectively.  In particular, $\Xi_k\setminus\Sigma\subset\Lambda_k^-\cup\{E\}$ and $\Sigma\setminus\Xi_k\subset\Lambda_k^-$, and the sets $\Xi_k$ and~$\Sigma$ agree on large equations in the sense that $\Xi_k\cap\Lambda_k^+=\Sigma\cap\Lambda_k^+$.
Next, observe that in the construction of~$\Xi_k$, the equations are removed from~$\Lambda_k$ in order of decreasing equation number, so at some intermediate stage one reaches the set $\Xi_k^{\text{step}}\coloneqq\Lambda_k^-\cup\{E\}\cup(\Lambda_k^+\cap\Xi_k)$.
What is the next equation to be removed at this step?
Since $\Xi_k$ only differs from $\Xi_k^{\text{step}}$ on low-numbered equations~$\Lambda_k^-$, the next equation that is removed is in that set.
However, before removing such a low-numbered equations, one should have removed~$E$, as the conjunction of $\Xi_k^{\text{step}}\setminus\{E\}$ is equivalent to that of $\Xi_k^{\text{step}}$: indeed, $E\not\in\Sigma$ hence
\begin{equation}
  \Sigma = (\Sigma\cap\Lambda_k^-)\cup(\Sigma\cap\Lambda_k^+)
  \subset \Lambda_k^- \cup (\Xi_k\cap\Lambda_k^+) = \Xi_k^{\text{step}}\setminus\{E\} ,
\end{equation}
so the conjunction of $\Xi_k^{\text{step}}\setminus\{E\}$ implies the conjunction of~$\Sigma$, which is in the equivalence class of interest.
Altogether, we have reached a contradiction with the construction of~$\Xi_k$, which ends the proof that early representatives are simply minimal for the total order~$\leq$.
\end{proof}

\subsection{Main search for proofs and counterexamples}
\label{subsec:main-search}

We now explain the script \texttt{scripts/semigroups-1.py} which found the $122$ equivalence classes of equations, and $456$ equivalence classes of conjunctions.
This script obtains ad-hoc representatives of the equivalence classes; it then produces data files with only the long representatives.  This data file is then organized as wanted by \texttt{scripts/semigroups-2-organize.py}

The first step concerns individual equations.
We define a wrapper function that attempts to resolve a given implication $E\implies E'$ by testing it against countermodels found so far, then attempting to prove it with a call to the Prover9 theorem prover, and if that fails attempting to disprove it with Mace4 (see \autoref{fig:prover9-mace4}).
When both approaches fail, a warning message is provided.  This happens especially often if a candidate implication is false but its smallest counterexample is either very large or infinite.  In the present project, whenever this occurred, a dedicated search provided a finite countermodel which was simply added to the list of known models to test.

\begin{figure}\centering
\begin{minipage}[t]{40ex}
\begin{lstlisting}[basicstyle=\ttfamily\scriptsize,linewidth=30ex,frame=single]
assign(sos_limit, 2000).
assign(max_weight, 1000).
assign(max_seconds, 1).

formulas(sos).
x=y*(x*y). x*(y*z)=(x*y)*z.

end_of_list.
formulas(goals).
x=x*x.

end_of_list.
\end{lstlisting}
\end{minipage}\quad
\begin{minipage}[t]{40ex}
\begin{lstlisting}[basicstyle=\ttfamily\scriptsize,linewidth=35ex,frame=single]
assign(selection_order, 2).
assign(selection_measure, 1).
assign(max_models, 1).
assign(max_seconds, 1).
assign(iterate_up_to, 100).

formulas(sos).
x=y*(x*y). x*(y*z)=(x*y)*z.

end_of_list.
formulas(goals).
x=x*x.

end_of_list.
\end{lstlisting}
\end{minipage}
\caption{\label{fig:prover9-mace4}Tool calls generated for the (false) implication from \Equareffull{14} (plus associativity) to \Equareffull{3}. \emph{Left:} Prover9. \emph{Right:} Mace4.}
\end{figure}

Some of the data is cached to avoid wasting seconds attempting to prove implications that are already known to be false or easily disproven by existing countermodels.
\begin{itemize}
\item a mapping from semigroup equations to minimal-numbered representatives of their equivalence class found so far;

\item a list of countermodels found (sometimes slowly) by Mace4, and a two-dimensional array stating for each model found so far and each (equivalence class of) semigroup equation found so far, whether the model satisfies the equation; this can be slow to compute for large models and equations with many variables;

\item a list of implications between these equivalence classes of individual equations.
\end{itemize}

At the end of this first step, everything is known about the equivalence classes and implication graph for individual semigroup equations of order up to~$4$.

\bigskip

Next, starting from the $122$ equations, for each associative theory found so far, we take its conjunction with any of the $122$ equations (restricted to having a larger equation number than any of those in the associative theory).  For each such conjunction we quickly check whether it is known to be equivalent to these same equations with one equation omitted.  If not, we must work harder: we use the aforementioned wrapper around Prover9 and Mace4 to check (and cache) which of the $122$ equations is implied by the conjunction, which allows us to very quickly evaluate whether the conjunction is equivalent to any of the conjunctions found previously.  If not, it is a new associative theory to be added to the list.
Eventually, the search stops with $456$ associative theory whose conjunctions with any of the $122$ equations has been found to be equivalent to other associative theories.  Along the way, we computed all individual equations implied by every associative theory, which is precisely what is needed to determine the sought-for long representatives.

\bigskip

We emphasize that this informal search gives insufficient correctness guarantees, as the Python script can have bugs or make incorrect assumptions on the behaviour of the external tools Mace4/Prover9.  Both happened in the course of this project:
\begin{itemize}
\item early on, success of Prover9 was not checked using the exit code, but by searching for the string ``\texttt{Exiting with 1 proof}'' in the output, whereas of course when one succeeds in proving two goals the output can be ``\texttt{Exiting with 2 proofs}'';
\item as explained at the end of \autoref{subsec:numbering}, constructing an internal representation of equations directly led to invalid objects and to claims of equivalence, which were only caught when failing to formalize them using the Vampire-to-Lean integration in \autoref{sec:lean}.
\end{itemize}

\subsection{Repository and data files}
\label{subsec:repository}

The scripts, Lean code and results of our exploration are available in the repository \url{https://github.com/blefloch/associative_theories/} which has the following structure.  Throughout the repository, \texttt{README.md} files describe the contents of the various directories in much more detail than here, as well as the relations between different files.

\begin{itemize}
\item \texttt{data/associative\_theories.json} is the main output of the project, a list of $456$ JSON objects each describing an associative theory, with properties
  \begin{itemize}
  \item \texttt{id}, index in the order~$\leq$ used to define
    \texttt{early} representatives in \autoref{prop:early-cutoff},
  \item \texttt{early}, early representative, as a list of ETP equation
    numbers including associativity \Equaref{4512},
  \item \texttt{long}, full list of ETP equation numbers implied by the
    equivalence class,
  \item \texttt{dual}, the \texttt{early} property of the dual
    associative theory.
  \end{itemize}

\item \texttt{associative\_theories.lean} is the main Lean file, which loads code imported from the Equational Theories Project (in \texttt{equational\_\allowbreak theories/}) that provides a few tools to manipulate equations, and code in \texttt{associative\_theories/} specific to this project; our strategy and files are described in detail in the next section.

\item \texttt{scripts/} has the Python scripts to generate all files in \texttt{data/}, and most of the contents of \texttt{associative\_theories/}; they rely on a separate installation of Prover9/Mace4 and Vampire.

\item \texttt{commentary/} features individual files with comments for many associative theories; none of it is formalized, and the list is not particularly complete as the comments were added by chance as we looked at specific theories.
\end{itemize}

\section{Lean formalization}
\label{sec:lean}

The Lean framework we use for the formalization is largely based on tools developed in the Equational Theories Project (ETP\@), most importantly the conversion from the Vampire automated theorem prover to Lean code.
For clarity, we isolate code from that project into the \texttt{equational\_theories} directory, which we treat as a library on the same footing as \texttt{mathlib}, while new code goes in a separate directory \texttt{associative\_theories} (see the \texttt{lakefile} in \autoref{fig:lakefile}).

\begin{figure}\centering
\fbox{
\begin{lstlisting}[basicstyle=\ttfamily\scriptsize]
name = "associative_theories"
defaultTargets = ["associative_theories"]

[leanOptions]
pp.unicode.fun = true
autoImplicit = false
relaxedAutoImplicit = false

[[require]]
name = "mathlib"
scope = "leanprover-community"
rev = "v4.20.0"

[[lean_lib]]
name = "equational_theories"

[[lean_lib]]
name = "associative_theories"
\end{lstlisting}}
\caption{\label{fig:lakefile}Configuration file \texttt{lakefile.toml}}
\end{figure}

The main theorems established in Lean are as follows.  The first states that for any of the $122$ equations a magma~$G$ satisfies associativity and this equation if and only if it satisfies the appropriate one of $456$ associative theories.
The second states that for any of the $456$ associative theories, and any of the $122$ equations, there exists an associative theory that is equivalent to their conjunction (for any semigroup~$G$).  The third states that for any pair of associative theories, if the (explicitly defined) function \texttt{AT\_impliesQ} gives True, then indeed the first associative theory implies the second.
\begin{center}
\begin{minipage}{76ex}
\begin{lstlisting}[language=lean,linewidth=74ex,frame=single]
theorem EQ_equiv (eqid : EQIndex) : ∀ (G : Type*) [Magma G],
  (Equation4512 G) ∧ (EQeq eqid G) <-> ATearly (EQ_to_AT eqid) G

theorem AT_conj (atid : ATIndex) (eqid : EQIndex) : ∃ (atid2 : ATIndex),
  ∀ (G: Type*) [Magma G], (ATearly atid G) ∧ (EQeq eqid G) <-> (ATearly atid2 G)

theorem AT_implies (atid: ATIndex) (atid2: ATIndex) :
  AT_impliesQ atid atid2 ->
  ∀ (G: Type*) [Magma G], (ATearly atid G) -> (ATearly atid2 G)
\end{lstlisting}
\end{minipage}
\end{center}

The Lean formalization splits into several tasks.
\begin{itemize}\setlength{\itemsep}{3pt plus 3pt}
\item \emph{Formalization skipped.} Prove that all equations of order $\leq 4$ are in a set of $653$ equations.  This has not been formalized in Lean.  The \texttt{laws\_complete} theorem from the ETP would be useful, but there are difficulties in defining and manipulating a map from the $4694$ equations of the ETP to their right-associated counterpart.

\item Prove that these $653$ equations are all equivalent to $122$.  This is done by establishing pairs of implications from each equation to its (lowest-numbered) representative using the Vampire--Lean integration.

\item Prove all single-equation implications: a minimal set of implications that generates the whole implication graph by transitivity is formalized using the Vampire--Lean integration, and transitivity is implemented by assembling these results by hand in \texttt{OneImpliDeduced.lean} to get all positive implications between single equations.

\begin{figure}\centering
\renewcommand{\leancomment}[1]{}
\fbox{
\begin{lstlisting}[firstline=10,lastline=21,language=lean]
theorem Equation3_56_4512_implies_Equation4 (G : Type*) [Magma G]
    (h3 : Equation3 G) (h56 : Equation56 G) (_ : Equation4512 G) : Equation4 G := by
  by_contra nh
  simp only [not_forall] at nh
  obtain ⟨sK0, sK1, nh⟩ := nh
  have eq12 (X0 : G) : (X0 ◇ X0) = X0 := mod_symm (h3 ..)
  have eq13 (X0 X1 : G) : (X0 ◇ (X1 ◇ (X1 ◇ X1))) = X0 := mod_symm (h56 ..)
  have eq15 : sK0 ≠ (sK0 ◇ sK1) := mod_symm nh
  have eq16 (X0 X1 : G) : (X0 ◇ (X1 ◇ X1)) = X0 := superpose eq12 eq13 -- forward demodulation 13,12
  have eq17 (X0 X1 : G) : (X0 ◇ X1) = X0 := superpose eq12 eq16 -- forward demodulation 16,12
  subsumption eq15 eq17
\end{lstlisting}}
\caption{\label{fig:vampireproventwo}Implication \Equaref{3}${}\wedge{}$\Equaref{56}${}\wedge{}$associativity$\implies$\Equaref{4} formalized by converting a Vampire proof to Lean using code from the Equational Theories Project}
\end{figure}

\item Additionally prove a minimal set of implications with multiple premises using the Vampire--Lean integration.  An example is given in \autoref{fig:vampireproventwo}.

\item Define $456$ conjunctions of equations, both through their early representatives and their long representatives, and prove their equivalence.

\item By pure logic manipulations (based on the equivalence of early and long), but with some speed difficulties, formalize the implication graph between associative theories, as a single theorem \texttt{AT\_implies}.

\item Then, the slowest part of the project to compile: for each associative theory and each equation, prove the theorem \texttt{AT\_conj} stating that their conjunction is equivalent to one of the associative theory in the list; as a corollary it implies that each equation on its own is equivalent to an associative theory.

\item \emph{Formalization skipped.} Prove that for any subset of the $653$ semigroup equations their conjunction is one of the associative theory.

\item \emph{Formalization skipped.} The above description concentrates on proving positive implications and equivalence, not proving that theories are really distinct.  The necessary countermodels are formalized (cf.\ the \texttt{Countermodels} directory and \autoref{fig:countermodel}), based on ETP tools to manipulate finite operation tables.  The structure of the files differs slightly from the ETP: given the centralized nature of this project, the countermodels are numbered sequentially instead of using their operation table as a name of the model and of the corresponding Lean theorem.  In this way the operation table is only specified once in the file.  However, the countermodels have not yet been used to formalize the disproof of implications.
\end{itemize}

\begin{figure}\centering
\fbox{\hbox to .88\textwidth{
\begin{lstlisting}[language=lean,linewidth=.9\textwidth,breaklines=true]
import equational_theories.Equations.All
import equational_theories.FactsSyntax
import equational_theories.MemoFinOp
import equational_theories.DecideBang
import Mathlib.Data.Finite.Prod

private def model10_op := finOpTable "[[0,0,2],[0,1,2],[0,2,2]]"

theorem facts_from_model10 :
  ∃ (G : Type) (_ : Magma G) (_: Finite G), Facts G
  [1, 3, 8, 47, 307, 323, 326, 333, 411, 3253, 3306, 3309, 3316, 3319, 3346, 3353, 4284, 4291, 4320, 4512]
  [2, 4, 5, 10, 11, 14, 16, 38, 39, 40, 41, 43, 56, 66, 75, 308, 309, 310, 311, 312, 313, 314, 315, 316, 318, 325, 327, 329, 332, 343, 419, 429, 440, 477, 504, 513, 3255, 3256, 3258, 3259, 3260, 3261, 3264, 3265, 3267, 3271, 3273, 3274, 3275, 3277, 3278, 3290, 3292, 3300, 3308, 3315, 3322, 3323, 3326, 3331, 3334, 3342, 3350, 3388, 3414, 4268, 4269, 4270, 4271, 4272, 4273, 4274, 4275, 4276, 4277, 4278, 4279, 4280, 4283, 4286, 4287, 4288, 4290, 4293, 4296, 4297, 4299, 4300, 4301, 4304, 4305, 4314, 4315, 4318, 4321, 4325, 4327, 4331, 4343, 4358, 4362, 4364, 4369]
  :=
  ⟨_, Magma.mk model10_op, Finite.of_fintype _, by decideFin!⟩

\end{lstlisting}\hss}}
\caption{\label{fig:countermodel}An example countermodel file.  Contrarily to the Equational Theories Project, the operation table $[[0,0,2],[0,1,2],[0,2,2]]$ (which states that $x\op 1=x$ and other $x\op y=y$) only appears once in the file.}
\end{figure}

All in all the main theorems are \texttt{AT\_implies} and \texttt{AT\_conj}, corresponding to the order and meet operation on the lattice of semigroup varities.

\section{Future work}
\label{sec:future}

The data files are sufficient to extract any information of interest on conjunctions, or implications between the laws, but it would be valuable to create visualization tools akin to the graph explorer developped in the Equational Theories Project.  Likewise, some commentary files were written in the course of the project, cf.~\texttt{commentary/} directory of \url{https://github.com/blefloch/associative_theories/}, but are currently not accessible in a user-friendly manner.
Some of the counterexamples likely have mathematical structure: for instance, models of the associative theory (cf.~\autoref{app:alltheories} for the full list of equations it implies)
\begin{equation}
  \textbf{AT-24:} \ \Equareffull{56}\ \wedge\ \Equareffull{75}\ \wedge\ \text{associativity}
\end{equation}
have sizes\footnote{The list of sizes was found by a Mace4 run rather than an abstract argument.} $1,3,9,27,81,\dots$, and the smallest non-commutative model has size~$27$.  Likewise, models of the associative theory
\begin{equation}
  \textbf{AT-62:} \ \Equareffull{440}\ \wedge\ \Equareffull{513}\ \wedge\ \text{associativity}
\end{equation}
have sizes $1,2,4,8,16,32,\dots$, and the smallest model that does not obey \Equareffull{504} has size~$32$.

In contrast to the case of magmas treated in the Equational Theories Project, we found no implication holding for finite sets but failing in general due to an infinite counterexample.  This raises the question of whether the situation would change for equations of order~$5$ or~$6$.
Equations of order up to~$5$ seem within reach: preliminary investigation shows that there are $380$ equivalence classes of equations up to that order, which is approximately three times as much as for equations of order up to~$4$.  Conjunctions might of course be significantly more numerous as the number of subsets of equations grows exponentially.

A much more difficult endeavour would be to determine all conjunctions of \emph{magma equations} of order up to~$4$.  This realm includes many interesting structures such as groups (equipped with division) or Boolean algebra.  Preliminary investigations imply a lower bound of several million equivalence classes.  The case of equations of order up to~$3$ is more reasonable: there are only a few thousand equivalence classes and the full implication semilattice (both finite and infinite) is almost determined, in work in progress by other members of the Equational Theories Project.

It would be interesting to compare the low-order results obtained in this work to the complete descriptions that are known for bands (idempotent semigroups) \cite{biryukov1970varieties,fennemore1970all,GERHARD1970195} and for commutative semigroups~\cite{kisielewicz1994varieties}.
Quasivarieties, defined by conditions of the form $\forall x,\dots,\forall y,(u_1=v_1\wedge\dots\wedge u_k=v_k)\implies u_0=v_0$ where $u_i,v_i$ are words in the variables $x,\dots,y$, are also very interesting~\cite{sapir1985lattice,adams2003quasivarieties} and worth exploring at low orders.

\begin{raggedright}
\bibliographystyle{plainurl}
\bibliography{references}
\par
\end{raggedright}

\appendix
\crefalias{section}{appendix}
\raggedbottom

\section{Supplementary material}
\label{app:supplementary}

\subsection{Equivalence classes of semigroup equations}
\label{app:allequations}

As explained in \autoref{subsec:numbering}, we adopt numbering conventions of the Equational Theories Project, keeping only the $653$ right-associated equations.
Here we list the $122$ equivalence classes of these equations, with the equivalences being understood assuming associativity.  The first equation in each line (namely, the lowest-numbered one) then serves as a representative of the whole class in \autoref{app:alltheories}.
The list here is extracted automatically from the file \texttt{data/associative\_equation\_classes.json} in the repository \url{https://github.com/blefloch/associative_theories/}

\begin{itemize}\setlength{\itemsep}{3pt}\raggedright
\ExplSyntaxOn
\cs_generate_variant:Nn \prop_set_from_keyval:Nn { Nx }
\file_get:nnN {associative_equation_classes.json} {} \l_tmpa_tl
\tl_replace_all:Nnn \l_tmpa_tl [ { { \if_false: } \fi: }
\tl_replace_all:Nnn \l_tmpa_tl ] { \if_false: { \fi: } }
\use:x { \exp_not:n { \clist_set:Nn \l_tmpa_clist } \l_tmpa_tl }
\clist_map_variable:NNn \l_tmpa_clist \l_tmpa_tl
  {
    \item
      \bool_set_false:N \l_tmpa_bool
      \clist_map_inline:Nn \l_tmpa_tl
      {
        \bool_if:NT \l_tmpa_bool { \ $\Leftrightarrow$\ }
        \bool_set_true:N \l_tmpa_bool
        \mbox{\Equareffull{#1}}
      }
  }
\end{itemize}

\subsection{List of associative equational theories}
\label{app:alltheories}

Out of these $122$ equivalence classes of equations, we list all equivalence classes of conjunctions.
This is given first as the minimal representative (as defined in \autoref{subsec:representative-choice}), then as the full set of equations (among the $122$ plus associativity) that are implied by the given theory.  In each case we indicate the dual theory (in the sense of exchanging operands of~$\op$), or indicate that the theory is self-dual.
The following list is extracted automatically from the file \texttt{data/associative\_theories.json} in the repository \url{https://github.com/blefloch/associative_theories/}

\begin{enumerate}\raggedright\setlength{\itemsep}{3pt}
\ExplSyntaxOn
\cs_generate_variant:Nn \prop_set_from_keyval:Nn { Nx }
\cs_generate_variant:Nn \clist_use:nn { x }
\cs_generate_variant:Nn \prop_const_from_keyval:Nn { cx }
\file_get:nnN {associative_theories.json} {} \l_tmpa_tl
\cs_set:Npn \tmp:w [#1] {#1}
\tl_set:Nx \l_tmpa_clist { \exp_last_unbraced:NV \tmp:w \l_tmpa_tl }
\int_new:N \l_my_id_int
\prop_new:N \l_early_to_id_prop
\clist_map_variable:NNn \l_tmpa_clist \l_tmpa_tl
  {
    \int_incr:N \l_my_id_int
    \tl_replace_all:Nnn \l_tmpa_tl " { }
    \tl_replace_all:Nnn \l_tmpa_tl : { = }
    \tl_replace_all:Nnn \l_tmpa_tl [ { { \if_false: } \fi: }
    \tl_replace_all:Nnn \l_tmpa_tl ] { \if_false: { \fi: } }
    \prop_const_from_keyval:cx { c_my_ \int_use:N \l_my_id_int _prop } \l_tmpa_tl
    \prop_put:Nxx \l_early_to_id_prop
      { \prop_item:cn { c_my_ \int_use:N \l_my_id_int _prop } { early } }
      { \int_use:N \l_my_id_int }
  }
\tl_new:N \l_my_early_tl
\tl_new:N \l_my_dual_tl
\int_step_inline:nn { \l_my_id_int }
  {
    \prop_set_eq:Nc \l_tmpa_prop { c_my_#1_prop }
    \item[{\bf AT-\prop_item:Nn \l_tmpa_prop { id }:}]
      \prop_get:NnN \l_tmpa_prop { early } \l_my_early_tl
      \prop_get:NnN \l_tmpa_prop { dual } \l_my_dual_tl
      \clist_set:Nx \l_tmpb_clist { \l_my_early_tl }
      \bool_set_false:N \l_tmpa_bool
      \clist_map_inline:Nn \l_tmpb_clist
        {
          \bool_if:NT \l_tmpa_bool { \ $\land$\ }
          \bool_set_true:N \l_tmpa_bool
          \int_compare:nTF { ##1 != 4512 }
            { \mbox{\Equareffull{##1}} }
            { associativity }
        }
      \str_if_eq:VVTF \l_my_early_tl \l_my_dual_tl
        { ,~self-dual }
        { ,~dual~to~ AT-\prop_item:No \l_early_to_id_prop { \l_my_dual_tl } }
      \\
      {\scriptsize\sl\{
      \clist_use:xn { \prop_item:Nn \l_tmpa_prop { long } } { ,~ }
      \}}
  }
\end{enumerate}

\end{document}